\documentclass[draft]{amsart}
\usepackage{amsfonts}
\usepackage{amssymb}
\usepackage{latexsym}
\usepackage{color}
\usepackage{euscript}
\usepackage{setspace}

\makeatletter
\@namedef{subjclassname@2010}{%
\textup{2010} Mathematics Subject Classification}
\makeatother

\theoremstyle{plain}
\newtheorem{thm}{Theorem}
\newtheorem{prop}[thm]{Proposition}
\newtheorem{cor}[thm]{Corollary}
\newtheorem{lem}[thm]{Lemma}

\theoremstyle{definition}

\theoremstyle{remark}

\newtheorem*{rem*}{{\it Remark}}

\newcommand{\Le}{\leqslant}
\newcommand{\Ge}{\geqslant}

\def\is#1#2{\langle#1,#2\rangle}

\def\pre#1{{#1}_\bot}
\def\dir#1{\int^\oplus_\varLambda #1(\lambda)\,\mu(\D\lambda)}

\DeclareMathOperator{\D}{d\!}  
\DeclareMathOperator{\trace}{tr}

\def\C{\mathbb{C}}

\def\Z{\mathbb{Z}}

\def\B{\mathbb B}
\def\N{\mathbb N}

\def\kk{\EuScript K}
\def\aa{\mathcal A}

\def\hh{\EuScript H}
\def\ff{\mathbf F}

\def\tt{\mathbf T}

\def\vv{\mathcal V}

\def\ss{\mathcal S}

\def\calw{{\mathcal W}}
\def\calv{{\mathcal V}}
\def\calt{{\mathcal T}}

\def\bsb{{\mathbf B}}

\def\bsd{{\mathbf D}}

\def\bor{\mathfrak{B}}
\def\H{\mathfrak{H}}
\def\K{\mathfrak{K}}

\title [A note on $k$-hyperreflexivity of Toeplitz-harmonic subspaces]{A note on $k$-hyperreflexivity of Toeplitz-harmonic subspaces}
\author{P.\ Budzy\'nski, K.\ Piwowarczyk \and M.\ Ptak}
\address{Piotr Budzy\'nski, Instytut Matematyczny Polskiej Akademii Nauk, ul.\ \'{s}w.\ Tomasza 30, 31-027 Krak\'ow, Poland}
\curraddr{Katedra Zastosowa\'n Matematyki, Uniwersytet Rolniczy w Krakowie, ul. Balicka 253c, 30-198 Krak\'ow, Poland}
    \email{piotr.budzynski@ur.krakow.pl}
\address{Kamila Piwowarczyk, Katedra Zastosowa\'n Matematyki, Uniwersytet Rolniczy w Krakowie, ul. Balicka 253c, 30-198 Krak\'ow, Poland}
    \email{kamila.piwowarczyk@ur.krakow.pl}
\address{Marek Ptak, Katedra Zastosowa\'n Matematyki, Uniwersytet Rolniczy w Krakowie, ul. Balicka 253c, 30-198 Krak\'ow, Poland}
    \email{rmptak@cyf-kr.edu.pl}
\thanks{The research of the first author was partially supported by the NCN (National Science Center) grant DEC-2011/01/D/ST1/05805.}
\onehalfspace
   \begin{document}
\begin{abstract}
The 2-hyperreflexivity of  Toeplitz-harmonic type subspace generated by an isometry or a quasinormal operator is shown. The $k$-hyperreflexivity of the tensor product $\ss\otimes \vv$ of a $k$-hyperreflexive decomposable subspace $\ss$ and an abelian von Neumann algebra $\vv$ is established.
\end{abstract}
\maketitle
\section{Introduction}
The concepts of reflexivity, transitivity and hyperreflexivity arise from the invariant subspace problem. An algebra  of operators is reflexive if it has so many (common) invariant subspaces that they determine the algebra itself or, equivalently, there are so many rank one operators in the preanihilator of the algebra that they characterize the algebra. The latter condition enables generalization of the reflexivity concept to subspaces of operators. The transitivity, in contrast to the reflexivity, means that there are no rank one operators in the preanihilator. A subspace of operators is hyperreflexive if the standard (norm) distance from any operator to the subspace is controlled by the distance induced by rank one operators (equivalently, in case of an algebra, the distance induced by invariant subspaces). The concepts  of $k$-reflexivity and $k$-hyperreflexivity are natural generalizations of reflexivity and hyperreflexivity (rank one operators are replaced by rank $k$ operators in relevant conditions). Certain subspaces, as being transitive, are far away from being reflexive. Nevertheless, they turn out to be 2-reflexive or 2-hyperreflexive. The space of all Toeplitz operators on the Hardy space on the unit disc is a primary example for this -- it is transitive (cf. \cite{az-pt1}) and $2$-hyperreflexive (cf. \cite{kl-pt1}). The same phenomenon occurs in the case of the space of all Toeplitz operators on the Hardy space on the polydisc (cf. \cite{pt1}).

The smallest weak${^*}$ closed subspace containing all powers of a given operator $A$ and all powers of its adjoint is called a Toeplitz-harmonic subspace generated by $A$. Clearly, the space $\calt(\mathbb{D})$ of all Toeplitz operators on the Hardy space $H^2(\mathbb{D})$ is a Toeplitz-harmonic subspace generated by the operator $T_z$ of multiplication by the independent variable. Hyperreflexivity of a Toeplitz-harmonic subspace generated by a $C_{00}$ contraction was shown in \cite{co-pt}. This result, in particular, implies hyperreflexivity of the Toeplitz-harmonic subspace generated by the Bergman operator $T_z$ -- the operator of multiplication by independent variable acting on the Bergman space on the unit disc. On the other hand, the Toeplitz-harmonic subspace generated by $T_z$ acting on $H^2(\mathbb{D})$ is transitive, thus it is not hyperreflexive. In view of this, it seems natural to ask about $k$-hyperreflexivity ($k\Ge 2)$ of a Toeplitz-harmonic subspace generated by an isometry or a quasinormal operator. In this paper we prove that such a subspace is 2-hyperreflexive (cf.\ Theorems \ref{012a} and \ref{023}). The proof is led via direct integral theory. For this purpose we prove  that tensor product of a $k$-hyperreflexive decomposable subspace which has property $\mathbb{A}_{1/k}(1)$ and an abelian von Neumann algebra is $k$-hyperreflexive (cf.\ Proposition \ref{015}).
\section{Preliminaries}
In all what follows, by $\Z$ and by $\N$ we denote the set of all integers and all non-negative integers, respectively.   Denote also by  $\hat\N$  the set $\N\cup\{\infty\}$.  If $X$ is a linear space and $Y$ is a subset of $X$, then $\text{lin }Y$ stands for the linear span of $Y$.

Suppose that $\hh$ is a (complex and separable) Hilbert space. Let $\bsb(\hh)$ denote the Banach algebra of all bounded linear operators on
$\hh$. For $k\in\N$, $\ff_k(\hh)$ stands for the set of operators on $\hh$ of rank at most
$k$. The set of  trace class operators on $\hh$ will be denoted by $\tt(\hh)$.  If $T\in\bsb(\hh)$, then $\calw(T)$ stands for the WOT closed algebra generated by $T$ and the identity operator $I$. For a family $\ss$ of operators in $\bsb(\hh)$, $\textrm{w$^*$-cl\,}\ss$  denotes the weak$^*$ closure of $\ss$.

Suppose that $\ss\subseteq\bsb(\hh)$ is a (linear) subspace. For an operator $A\in\bsb(\hh)$ and $k\in\N$ we consider the following
quantities
    \begin{align*}
    d(A,\ss)=\inf\{\|A-T\|\colon T\in\ss\},\notag\quad \alpha_k(A,\ss)=\sup\{|\is{A}{t}| \colon t\in\pre{\ss}\cap\B_k(\hh)\},
    \end{align*}
where $\is{A}{t}=\trace(At)$, $\pre{\ss}=\{t\in\tt(\hh)\colon \is{T}{t}=0 \text{ for all }T\in\ss\}$ and $\B_k(\hh)$ stands for the unit ball in $\ff_k(\hh)$
(with respect to trace norm $\|\cdot\|_1$). Recall that $d(A,\ss)\Ge \alpha_k(A,\ss)$ for every $A\in\bsb(\hh)$. The subspace $\ss$ is called {\em $k$-hyperreflexive} if there is a constant $C$ such that
    \begin{align}\label{001}
    d(A,\ss)\Le C\,\alpha_k(A,\ss),\quad A\in\bsb(\hh).
    \end{align}
By $\kappa_k(\ss)$ we denote the infimum of the collection of all constants $C$ such that \eqref{001} holds. An operator $T\in\bsb(\hh)$ is said to be $k$-hyperreflexive if $\calw(T)$ is $k$-hyperreflexive. For more background of reflexivity and $k$-hyperreflexivity see \cite{co} and \cite{kl-pt1}.

It is known that $k$-hyperreflexivity is not hereditary in general, i.e., a subspace of $k$-hyperreflexive subspace need not to be $k$-hyperreflexive itself. The situation improves if the subspace has property $\mathbb{A}_{1/k}(r)$ (cf. \cite[Proposition 3.8]{kl-pt1}). Let $k\in\mathbb{N}$ and $r\in[1,\infty)$. Recall that a weak$^*$ closed subspace $\ss$ of $\bsb(\hh)$ has property $\mathbb{A}_{1/k}(r)$ if for every
weak$^*$ continuous functional $\psi\colon\bsb(\hh)\to\C$ and every $\varepsilon>0$ there is $f\in\ff_k(\hh)$ such that
$\|f\|_1\Le (r+\varepsilon)\|\psi\|_1$ and $\psi(T)=\is{T}{f}$ for all $T\in\ss$. Direct integral of subspaces which have property
$\mathbb{A}_{1/k}(1)$ has this property as well (cf.\ \cite[Theorem 3.6]{ha-no} and Lemma \ref{011+}). We will frequently use the following fact, which is a generalization of a result due to Kraus and Larson in \cite[Theorem 3.3]{kra-lar}.
\begin{lem}\label{hyp-sub}\mbox{\cite[Proposition 3.8]{kl-pt1}}
Let $k\in\N$ and $r\in[1,\infty)$. If $\ss$ is a $k$-hyper\-reflexive subspace of $\bsb(\hh)$ which has property $\mathbb{A}_{1/k}(r)$, then any weak$^*$ closed subspace $\ss_1$ of $\ss$ is $k$-hyperreflexive and $\kappa_k(\ss_1)\Le r + (r+1)\kappa_k(\ss)$.
\end{lem}

Let us now recall some basic definitions concerning direct integrals of subspaces (we refer the reader to monographs \cite{di} and \cite{sch} for more information on the direct integral theory). Let $\big(\varLambda,\bor, \mu\big)$ be a measure space, where $\varLambda$ is a separable metric space, $\bor$ is the $\sigma$-algebra of all Borel subsets of $\varLambda$ and $\mu$ is a $\sigma$-finite regular Borel measure on $\varLambda$. Let $\hh_1\subseteq\hh_2\subseteq\ldots\subseteq\hh_\infty$ be a sequence of Hilbert spaces. Suppose that $\{\varLambda_n\colon n\in\hat\N\}\subseteq\bor$ is a partition of $\varLambda$. Let $\hh(\lambda)=\hh_n$ for $\lambda\in\varLambda_n$ and $n\in\hat\N$. Then $\H=\dir{\hh}$
denotes the Hilbert space of all (equivalence classes of) $\bor$-measurable $\hh_\infty$-valued functions on $\varLambda$ such that for $\mu$-a.e. $\lambda\in\varLambda$, $f(\lambda)\in\hh(\lambda)$ and $\int_\varLambda\|f(\lambda)\|^2\mu(\D\lambda)<\infty$. Throughout the rest of the paper,
$\bsd(\H)$ (resp. $\bsd'(\H)$) will stand for the set of all diagonal (resp. decomposable) operators in $\bsb(\H)$ ($\bsd'(\H)$ is indeed a commutant of $\bsd(\H)$; cf.\ \cite[Lemma I.3.2]{sch}). Suppose that $\ss$ is an weak$^*$ closed subspace of $\bsd'(\H)$ such that there exists a countable generating set $\{T_{n}\colon n\in\N\}$ for $\ss$. Such a subspace $\ss$ is said to be {\em decomposable}. For $\lambda\in\varLambda$, let $\ss(\lambda)$ be the weak$^*$ closed subspace generated by $\{T_n(\lambda)\colon n\in\N\}$. It is a matter of verification that definition of $\ss(\lambda)$ does not depend on the choice of a generating set (cf. \cite[p. 1397]{kl-pt3}). The family $\{\ss(\lambda)\colon \lambda\in\varLambda\}$ is called {\em decomposition} of $\ss$. For decomposable $\ss$ we define the subspace $\ss_\bsd$ by
    \begin{align*}
    \ss_\bsd=\textrm{w$^*$-cl\,}\textrm{lin} \{D\, T: D \in \bsd(\H), T \in \ss \};
    \end{align*}
it is an analog of the algebra $\dir{\aa}$ appearing in the reduction theory of von Neumann algebras.
\section{$k$-hyperreflexivity}
We begin our investigations of $k$-hyperreflexivity with variants of results due to Hadwin (cf. \cite[Theorems 3.8 and 6.16]{ha}), Hadwin and Nordgren (cf. \cite[Theorem 3.6]{ha-no}). Originally, they concerned hyperreflexivity and property $\mathbb{A}_{1}(r)$ but the claims are valid for $k$-hyperreflexivity and $\mathbb{A}_{1/k}(r)$, respectively, as well (here, and later, ``hyperreflexivity'' means ``$1$-hyperreflexivity''). Since the results can be proved in a very similar fashion to the original ones, we omit the proofs.
\begin{lem}\label{011+}
Let $k\in\N$, $r\in[1,\infty)$ and $K\in (0,\infty)$. Suppose $\ss$ is a decomposable and $\{\ss(\lambda)\colon\lambda\in\varLambda\}$ is its
decomposition. Then the following holds.
\begin{enumerate}
\item[(i)] If, for $\mu$-a.e. $\lambda \in \varLambda$, $\ss(\lambda)$ is $k$-hyperreflexive and $\kappa_k(\ss(\lambda)) \leq K$, then $\ss_\bsd$ is $k$-hyperreflexive and $\kappa_k(\ss_\bsd) \leq 2+3K$.
\item[(ii)] If, for $\mu$-a.e. $\lambda\in\varLambda$, $\ss(\lambda)$ has property $\mathbb{A}_{1/k}(r)$, then $\ss_\bsd$ has property $\mathbb{A}_{1/k}(r)$.
\end{enumerate}
\end{lem}
If $\calv$ is an abelian von Neumann algebra, then by the reduction theory (cf. \cite[Theorem I.2.6]{sch}), $\calv$ is (up to unitary equivalence) the diagonal algebra $\bsd(\K)$ corresponding to a direct integral decomposition of a underlying Hilbert space $\kk$. The algebra $\mathbb{C} I$ is hyperreflexive, $\kappa_1(\calv)\Le 1$ (cf. \cite[Proposition 3.11]{kra-lar}) and it has property $\mathbb{A}_1(1)$ (cf. \cite[Proposition 60.1]{co}). Hence, by Lemma \ref{011+}, we get the following (the part concerning hyperreflexivity is contained in \cite[Theorem 3.5]{ro}).
\begin{cor}\label{abelian-hyp}
Let $\calv$ be an abelian von Neumann algebra. Then $\calv$ is hyperreflexive, $\kappa_1(\calv)\Le 5$ and it has property $\mathbb{A}_1(1)$.
\end{cor}
The next lemma is an analog of \cite[Proposition II.3.4]{di}. It reveals a relation between direct integrals of a constant (up to the unitary equivalence) field of subspaces and tensor products. The proof is essentially the same as of the original result so we left it to the reader. A piece of notation is required: if $\mathcal{V}\subset \bsb(\hh)$ and $\mathcal{N}\subset \bsb(\kk)$ are weak$^*$ closed subspaces, then $\mathcal{V}\otimes\mathcal{N}$ denotes the weak$^*$ closed subspace  of $\bsb(\hh\otimes\kk)$ generated by the set $\{A\otimes B\colon A\in\mathcal{V}, B\in\mathcal{N}\}$.
\begin{lem}\label{013}
Let $\ss_1$ be a weak$^*$ closed subspace of $\bsb(\hh)$. Let $\ss\subset\bsb(\H)$ be a de\-com\-posable subspace and $\{\ss(\lambda)\colon\lambda\in\varLambda\}$ be its decomposition.  Suppose, for every $\lambda\in\varLambda$, there is a unitary operator
$U(\lambda)\in\bsb(\hh,\hh(\lambda))$ such that $U(\lambda) \ss_1  U(\lambda)^{-1}=\ss(\lambda)$. Then there is a unitary operator
$U\colon \H\otimes\hh\to\H$ such that $U\big(\bsd(\H)\otimes \ss_1\big)U^{-1}=\ss_\bsd$.
\end{lem}
We are now ready to prove $k$-hyperreflexivity of the tensor product of some sub\-spaces. This is related to a result due to S. Rosenoer concerning
hyperreflexivity of the tensor product of a hyperreflexive von Neumann algebra and the algebra of all analytic Toeplitz operators on the Hardy space $H^2(\mathbb{D})$ (cf. \cite[Theorem 2]{ro2}).
\begin{prop}\label{015}
Let $\vv\subseteq\bsb(\kk)$ is an abelian von Neumann algebra. Let $\ss\subset\bsb(\H)$ be a de\-com\-posable subspace and $\{\ss(\lambda)\colon\lambda\in\varLambda\}$ be its decomposition. If $\ss$ is $k$-hyperreflexive  and has property $\mathbb{A}_{1/k}(1)$, then every weak$^*$ closed subspace $\ss_1$ of $\ss\otimes \vv$ is $k$-hyperreflexive and $\kappa_k(\ss_1)\Le 5+6\kappa_k(\ss)$. If $\ss_1=\ss\otimes \vv$, then $\kappa_k(\ss_1)\Le 2+3\kappa_k(\ss)$.
\end{prop}
\begin{proof}
By the reduction theory (cf. \cite[Theorem I.2.6]{sch}), $\vv$ is (up to a unitary equivalence) the diagonal algebra $\bsd(\K)$ corresponding to a direct integral decomposition of the underlying Hilbert space $\kk$. In view of Lemmas \ref{011+} and \ref{013} the subspace $\ss\otimes\bsd(\K)$ is $k$-hyperreflexive and $\kappa_k(\ss\otimes\bsd(\K))\Le 2+3\kappa_k(\ss)$. Lemmas \ref{011+} and \ref{013} imply that $\ss\otimes\bsd(\K)$ has property $\mathbb{A}_{1/k}(1)$. If $\ss_1\subseteq\ss\otimes \vv$, then $\ss_1$ is $k$-hyperreflexive and $\kappa_k(\ss_1)\Le 5+6\kappa_k(\ss)$ by Lemma \ref{hyp-sub}.
\end{proof}
\begin{cor}\label{022}
Let $\calt(\mathbb{D})$ be weak$^*$ closed subspace of all Toeplitz operators on the Hardy space on the unit disc and $\vv$ be an abelian von Neumann
algebra. Then every weak$^*$ closed subspace $\calt_1$ of $\calt(\mathbb{D})\otimes \vv$ is $2$-hyperreflexive, $\kappa_2(\calt_1)\Le 17$ and it has property $\mathbb{A}_{1/2}(1)$. If $\calt_1=\calt(\mathbb{D})\otimes \vv$, then $\kappa_2(\calt_1)\Le 8$
\end{cor}
\begin{proof}
By \cite[Theorem 4.1, Proposition 4.2]{kl-pt1}, $\calt(\mathbb{D})$ is $2$-hyperreflexive, $\kappa_2(\calt(\mathbb{D}))\Le2$ and $\calt(\mathbb{D})$ has property $\mathbb{A}_{1/2}(1)$. Therefore, by Proposition \ref{015}, $\calt(\mathbb{D})\otimes \calv$ is 2-hyperreflexive, $\kappa_2(\calt(\mathbb{D})\otimes \vv)\Le 8$ and it has property $\mathbb{A}_{1/2}(1)$. If $\calt_1\subseteq \calt(\mathbb{D})\otimes \vv$, then, in view of  Lemma \ref{hyp-sub}, $\calt_1$ is 2-hyperreflexive and $\kappa_2(\calt_1)\Le 17$.
\end{proof}
\begin{cor}
Let $\mathcal{N}$ be a von Neumann algebra with an abelian commutant and let $\vv$ be an abelian von Neumann algebra. Assume that $\mathcal{N}$ has property $\mathbb{A}_{1/k}(1)$. Then every is weak$^*$ closed subspace $\ss$ of $\mathcal{N}\otimes \vv$ is $k$-hyperreflexive, $\kappa_k(\ss)\Le 17$ and it has property $\mathbb{A}_{1/k}(1)$. If $\ss=\mathcal{N}\otimes \vv$, then $\kappa_k(\ss)\Le 8$.
\end{cor}
\begin{proof}
By \cite[Lemma 3.1]{ro} the algebra $\mathcal{N}$ is $k$-hyperreflexive and $\kappa_k(\mathcal{N})\Le2$. Since it has property $\mathbb{A}_{1/k}(1)$, $\mathcal{N}\otimes \vv$ is k-hyperreflexive and $\kappa_k(\mathcal{N}\otimes \vv)\Le8$ by Proposition \ref{015}. If $\ss\subseteq\mathcal{N}\otimes \vv$, then $\ss$ is $k$-hyperreflexive and $\kappa_k(\ss)\Le 17$ by Lemma \ref{hyp-sub}.
\end{proof}
Now we turn our attention to the question of $k$-hyperreflexivity of Toeplitz-harmonic subspaces. It is well-known that the space $\calt(\mathbb{D})$ of all Toeplitz operators acting in the Hardy space $H^2(\mathbb{D})$ is a weak$^*$ closed subspace generated by $T_z^n$, ${T_z^*}^m$, $n,m\in\mathbb{N}$, i.e.,
    \begin{align*}
    \calt(\mathbb{D})=\textrm{w$^*$-cl\,}\{p(T_z)+q(T_z)^*\colon \ p\ \text{and}\ q\  \text{are analytic polynomials}\},
    \end{align*}
where $T_z$ is the operator of multiplication by the independent variable acting on $H^2(\mathbb{D})$. For a given $A\in\bsb(\hh)$, a {\em Toeplitz-harmonic subspace} generated by $A$ is defined by
    \begin{align*}
    \calt(A)=\textrm{w$^*$-cl\,}\{p(A)+q(A)^*\colon \ p\ \text{and}\ q\  \text{are analytic polynomials}\}.
    \end{align*}
Clearly, $\calw(A)\subseteq\calt(A)$.
We have  $\calt(H^2(\mathbb{D}))=\calt(T_z)$, and $\calw(T_z)\subseteq\calt(T_z)$. Recall that $\calw(T_z)$ is hyperreflexive \cite[Theorem 2]{da}. On the other hand, $\calt(T_z)$ is transitive \cite[Theorem 3.1]{az-pt1}, $2$-hyperreflexive (cf.\ \cite[Corollary 4.2]{kl-pt1}) and has property $\mathbb{A}_{1/2}(1)$ (cf.\ \cite[Theorem 4.1]{kl-pt1}). Furthermore, for an isometry $V$, the algebra $\calw(V)$ is also hyperreflexive \cite[Corollary 6]{kl-pt2}. The question arises naturally whether $\calt(V)$ is 2-hyperreflexive. As shown below, the answer is in the affirmative.
\begin{thm}\label{012a}
Let $V\in\bsb(\hh)$ be an isometry. Then every weak$^*$ closed subspace $\ss_1$ of $\calt(V)$ is $2$-hyperreflexive, $\kappa_2(\ss_1)\Le 35$ and it has property $\mathbb{A}_{1/2}(1)$.
\end{thm}
It turns out that the same conclusion holds for a Toeplitz-harmonic subspace generated by quasinormal operator. Both the results are proved in the same way.
\begin{thm}\label{023}
Suppose that $T\in\bsb(\hh)$ is a quasinormal operator. Then every weak$^*$ closed subspace $\ss_1$ of $\calt(T)$ is $2$-hyperreflexive, $\kappa_2(\ss_1)\Le 35$ and it has property $\mathbb{A}_{1/2}(1)$.
\end{thm}
\begin{proof}
By Brown's result \cite[Theorem 1]{brown} every quasinormal operator is unitarily equivalent to $N\oplus(A\otimes S)$, where $N$ is normal, $S$ is the unilateral shift operator and $A$ is positive (if $T=V$ is an isometry, then $A=I$). Since $k$-hyperreflexivity is kept with the same constant by unitary equivalence it is sufficient to consider the above model. Furthermore, $\ss_1\subseteq\calt(T)\subseteq\mathcal{N}(N) \oplus\calw(A)\otimes \calt(S)$, where $\mathcal{N}(N)$ denotes the smallest (abelian) von Neumann algebra containing $N$ and the identity operator $I$. Hence, it suffices to show that subspaces $\mathcal{N}(N)$ and $\calw(A)\otimes \calt(S)$ are $2$-hyperreflexive and have property $\mathbb{A}_{1/2}(1)$.

Since $\mathcal{N}(N)$ is a abelian von Neumann algebra, by Corollary \ref{abelian-hyp} (or \cite[Theorem 60.14]{co}), it is hyperreflexive, $\kappa_1(\mathcal{N}(N))\Le5$ and it has property $\mathbb{A}_{1/2}(1)$. As a consequence it is $2$-hyperreflexive and $\kappa_2(\mathcal{N}(N))\Le5$.

The algebra $\calw(A)$ is an abelian von Neumann algebra. Since $S$ is unitarily equivalent to $T_z$, we see that subspaces $\calt(S)$ and $\calt(T_z)=\calt(\mathbb{D})$ are unitarily equivalent. By Corollary \ref{022}, the tensor product $\calw(A)\otimes\calt(S)$ is 2-hyperreflexive with constant $\kappa_2(\calw(A)\otimes\calt(S))\leqslant 8$ and has property $\mathbb{A}_{1/2}(1)$.

In view of \cite[Corollary  5.3]{kl-pt1} all of the above yields $2$-hyperreflexivity of $\mathcal{N}(N)\oplus \calw(A)\otimes \calt(S)$ with the constant less or equal to $17$. Hence $\ss_1$ is $2$-hyperreflexive as being a subspace of the latter (see Lemma \ref{hyp-sub}). Moreover, we have $\kappa_2(\ss_1)\Le 35$.
\end{proof}
\bibliographystyle{amsalpha}

\end{document}